\documentclass[11pt]{article}
\usepackage{amsmath,latexsym,amsfonts,amssymb,mathrsfs}
\usepackage{graphicx,verbatim,color}
\usepackage[ruled,vlined]{algorithm2e}

\newtheorem{theorem}{Theorem}[section]
\newtheorem{lemma}[theorem]{Lemma}

\newenvironment{proof}{{\noindent \bf
Proof:}}{\hfill\qed\bigskip}

\newenvironment{proofof}[1]{\vspace{0.2cm}\hfill\\%
\noindent\textbf{Proof of #1.}}{$\Box$
\vspace{0.2cm}\\}

\newcommand{\mS}{\mathbb S}
\newcommand{\R}{\mathbb R}

\newcommand{\cA}{{\cal A}}

\newcommand{\cF}{{\cal F}}
\newcommand{\cG}{{\cal G}}

\newcommand{\cN}{{\cal N}}
\newcommand{\cP}{{\cal P}}

\newcommand{\bn}{{\bf n}}
\newcommand{\bs}{{\bf s}}
\newcommand{\bp}{{\bf p}}
\newcommand{\bq}{{\bf q}}
\newcommand{\bx}{{\bf x}}
\newcommand{\by}{{\bf y}}
\newcommand{\bz}{{\bf z}}

\newcommand{\bsomega}{{\boldsymbol \omega}}

\newcommand{\qed}{\hfill$\Box$}

\newcommand{\be}{\begin{equation}}
\newcommand{\ee}{\end{equation}}
\newcommand{\ba}{\begin{array}}
\newcommand{\ea}{\end{array}}
\newcommand{\beas}{\begin{eqnarray*}}
\newcommand{\eeas}{\end{eqnarray*}}
\newcommand{\bea}{\begin{eqnarray}}
\newcommand{\eea}{\end{eqnarray}}
\newcommand{\lb}{\label}
\newcommand{\sumlk}{\sum_{\ell=0}^{\infty} \sum_{k=1}^{N(d,\ell)}}

\newcommand{\wh}{\widehat}
\newcommand{\whphi}{\widehat{\phi}}
\newcommand{\inprod}[2]{\left\langle{#1},{#2}\right\rangle}
\newcommand{\vertiii}[1]{{\left\vert\kern-0.25ex\left\vert\kern-0.25ex\left\vert #1
    \right\vert\kern-0.25ex\right\vert\kern-0.25ex\right\vert}}
\newcommand{\HsigS}{{H^\sigma(\mS^d)}}
\newcommand{\HsigO}{{H^\sigma(\Omega)}}

\title{Zooming from Global to Local: A Multiscale RBF Approach}
\author{Q.~T.~Le Gia, I.~H.~Sloan \\ School of Mathematics and Statistics\\
           University of New South Wales\\ Sydney, NSW 2052\\
           Australia
\and H.~Wendland \\ Department of Mathematics\\ University of
Bayreuth\\ D-95440 Bayreuth\\ Germany}
\begin{document}

\maketitle

\begin{abstract}
Because physical phenomena on Earth's surface occur on many different
length scales, it makes sense when seeking an efficient approximation to
start with a crude global approximation, and then make a sequence of
corrections on finer and finer scales. It also makes sense eventually to
seek fine scale features locally, rather than globally. In the present
work, we start with a global multiscale radial basis function (RBF)
approximation, based on a sequence of point sets with decreasing mesh
norm, and a sequence of (spherical) radial basis functions with
proportionally decreasing scale centered at the points. We then prove that
we can ``zoom in'' on a region of particular interest, by carrying out
further stages of multiscale refinement on a local region. The proof
combines multiscale techniques for the sphere from Le Gia, Sloan and
Wendland, SIAM J. Numer. Anal. 48 (2010) and Applied Comp. Harm. Anal. 32
(2012), with those for a bounded region in $\mathbb{R}^d$ from Wendland,
Numer. Math. 116 (2012). The zooming in process can be continued
indefinitely, since the condition numbers of matrices at the different
scales remain bounded. A numerical example illustrates the process.
\end{abstract}

\section{Introduction}

In many modern areas of geosciences, such as geomagnetic or gravitational
field modeling, the problem of interpolation from scattered data on the
sphere arises naturally. Such problems are often of multiscale nature, so one
would like models that can be used to draw conclusions globally as well as
locally. For example, in modeling the global gravitational field one
should be able to see the general nature of the global field as well as
local gravitational anomalies.

Multiscale interpolation and approximation for functions on the unit
sphere has been considered by a number of authors using different
techniques. Some authors used wavelets defined on spheres
\cite{AntVan99,SchSwe95}, but these are not suitable for scattered data.
Other authors have proposed kernel methods based on truncations of the
expansions of some special kernels into spherical harmonics
\cite{FreGerSch98,NarPetWar06,Mha05}; these methods require a quadrature
scheme on the sphere which can integrate spherical polynomials exactly,
but the construction of a good quadrature based on scattered data is
itself a non-trivial problem
\cite{GraKunPot09,HSW_handbook,LeGMha08,FilThe08,MhaNarWar01}.

In recent articles \cite{LeGSloWen10,LeGSloWen12}, we proposed a
multiscale interpolation framework using radial basis functions for
functions that lie in Sobolev spaces defined on the unit sphere. The
theory underlying our multiscale method will work for scattered data. In
this paper we introduce a ``zooming-in'' framework, which allows the
multiscale algorithm to model the data from the global scale and then zoom
in to local regions.  We do this by combining multiscale techniques for
the sphere with those for a bounded region established in \cite{Wen10}.

The paper is organised as follows. In Section~\ref{sec:prelim} we review
necessary materials about Sobolev spaces on spheres and positive definite
kernels defined via radial basis functions (RBFs). The global and local
multiscale algorithm using spherical RBFs is then introduced in
Section~\ref{sec:zoom}. A convergence results for functions in the native
space in given there. The next section, Section~\ref{sec:escape}, deals
with convergence results for function in Sobolev spaces with lesser
smoothness. Finally we conclude the paper with numerical experiments given
in Section~\ref{sec:numerics}.
\section{Preliminaries}\label{sec:prelim}
In this section, we will introduce necessary materials for the main results presented
in the paper.
\subsection{Sobolev spaces on the unit sphere}
Let $\mS^d$ be the unit sphere in $\R^{d+1}$.
Denote the inner product in $L_2(\mS^d)$ by
\[
\inprod{v}{w} :=  \int_{\mS^d} v w dS,
\]
where $dS$ is the surface measure on the unit sphere, and denote the
measure of the whole sphere by $\omega_d$ (so, for example,
$\omega_2=4\pi$). Recall \cite{Mul66} that a spherical harmonic is the
restriction to $\mS^d$ of a homogeneous harmonic polynomial $Y(\bx)$ in
$\R^{d+1}$. The space of all spherical harmonics of degree $\ell$ on
$\mS^d$, denoted by $H_\ell$, has an $L_2$ orthonormal basis
\[
  \{ Y_{\ell,k}:\; k=1,\ldots,N(d,\ell)\},
\]
where
\[
N(d,0)=1 \mbox{ and } N(d,\ell)=\frac{(2\ell+d-1)\Gamma(\ell+d-1)}
{\Gamma(\ell+1)\Gamma(d)} \text{ for } \ell \ge 1.
\]
The space of spherical harmonics of degree $\le L$ will be denoted
by $\cP_L := \oplus_{\ell=0}^L H_\ell$; it has dimension $N(d+1,L)$.

Every function $f\in L_2(\mS^d)$ can be expanded in terms of
spherical harmonics,
\[
 f = \sum_{\ell=0}^\infty \sum_{k=1}^{N(d,\ell)} \widehat{f}_{\ell, k} Y_{\ell, k},
 \qquad
 \widehat{f}_{\ell, k} = \inprod{f}{Y_{\ell, k}}.
\]
For a non-negative parameter $\sigma$, the Sobolev space $H^\sigma(\mS^d)$ may
be defined by
\begin{equation}\label{def:Sob S}
H^\sigma(\mS^d):= \left\{ f \in L_2(\mS^d):
\|f\|^2_{H^\sigma(\mS^d)} := \sum_{\ell=0}^\infty \sum_{k=1}^{N(d,\ell)}
(1+\ell)^{2\sigma} | \widehat{f}_{\ell, k}|^2 < \infty \right\}.
\end{equation}
Note that $H^0(\mS^d) = L_2(\mS^d)$.


Sobolev spaces on $\mS^d$ can also be defined using local charts
(see \cite{LioMag}). Here we use a specific atlas
of charts, as in \cite{HubMor04}.

Let $\bz$ be a given point on $\mS^d$.  The spherical cap centered at
$\bz$ of radius $\theta$ is defined by
\begin{eqnarray*}
G(\bz,\theta) &:=& \{ \by \in \mS^{d} : \cos^{-1}(\bz \cdot \by) < \theta \},  \qquad \theta \in (0, \pi),
\end{eqnarray*}
where $\bz\cdot \by$ denotes the Euclidean inner product of $\bz$~and $\by$
in $\R^{d+1}$.

Let $\hat{\bn}$ and $\hat{\bs}$ denote the north
and south poles of $\mS^d$, respectively. Then a simple cover
for the sphere is provided by
\be\lb{def_Ui}
  U_1 = G(\hat{\bn},\theta_0) \quad \mbox{and}\quad
  U_2 = G(\hat{\bs},\theta_0),
  \mbox{ where } \theta_0 \in (\pi/2,2\pi/3).
\ee
The stereographic projection $\sigma_{\hat{\bn}}$ of the punctured sphere
$\mS^d\setminus\{\hat{\bn}\}$ onto $\R^d$ is defined as a mapping that maps
$\bx \in \mS^d \setminus \{\hat{\bn}\}$ to the intersection of the equatorial
hyperplane $\{\bz=0\}$ and the extended line that passes through $\bx$ and $\hat{\bn}$. 
The stereographic projection $\sigma_{\hat{\bs}}$ based on $\hat{\bs}$ can be defined analogously. We set
\be\lb{def_phi}
  \psi_1 = \frac{1}{\tan(\theta_0/2)}\sigma_{\hat{\bs}}|_{U_1}
  \quad \mbox{ and }\quad
  \psi_2 = \frac{1}{\tan(\theta_0/2)}\sigma_{\hat{\bn}}|_{U_2},
\ee
so that $\psi_k$, $k=1,2$, maps $U_k$ onto $B(0,1)$, the unit ball in $\R^d$.
We conclude that $\cA =\{U_k,\psi_k\}_{k=1,2}$ is a $C^\infty$ atlas of
covering coordinate charts for the sphere. It is known (see \cite{Rat94})
that the stereographic coordinate charts $\{\psi_k\}_{k=1,2}$ as defined in
(\ref{def_phi}) map spherical caps to Euclidean balls, but in general
concentric spherical caps are not mapped to concentric Euclidean balls.
The projection $\psi_k$, for $k=1,2$, does not distort too much the geodesic
distance between two points $\bx,\by \in \mS^d$, as shown in \cite{LeGNarWarWen06}.

With the atlas so defined, we define the map $\pi_k$ which takes a real-valued
function $g$ with compact support in $U_k$ into a real-valued function on $\R^d$ by
\[
 \pi_k(g) (\bx) = \left\{  \begin{array}{ll}
                            g \circ \psi^{-1}_k(\bx),&
                            \mbox{ if } \bx\in B(0,1),\\
                            0, & \mbox{ otherwise }.
                         \end{array}
                \right.
\]
Let $\{\chi_k:\mS^d \rightarrow \R \}_{k=1,2}$ be a partition of unity subordinated to the atlas, 
i.e., a pair of non-negative infinitely differentiable functions $\chi_k$ on $\mS^d$ with compact support in $U_k$,
such that $\sum_{k}\chi_k = 1$. For any function $f:\mS^d \rightarrow \R$, we can use the partition of unity to write
\[
  f = \sum_{k=1}^2 \chi_k f,
  \mbox{ where } (\chi_k f)(\bx) = \chi_k(\bx)f(\bx), \quad \bx \in \mS^d.
\]
The Sobolev space $H^\sigma(\mS^d)$ is then the set
\[
  \left\{f \in L_2(\mS^d) :
  \pi_k(\chi_k f) \in H^\sigma(\R^d) \quad\mbox{ for } k=1,2 \right\},
\]
which is equipped with the norm
\begin{equation}\label{defH1atlas}
  \vertiii {f}_{H^\sigma(\mS^d)} :=
  \left( \sum_{k=1}^2 \|\pi_k (\chi_k f)\|^2_{H^\sigma(\R^d)} \right)^{1/2}.
\end{equation}
This norm is equivalent to the $H^\sigma(\mS^d)$ norm given in
\eqref{def:Sob S} (see \cite{LioMag}). 
From now on we will use only the $\Vert \cdot \Vert$ notation for the equivalent norms.

We recall that \cite{Ada75} the Sobolev space $H^\sigma(\R^d)$
is the set
\[ 
\left\{ f \in L_2(\R^d) : \int_{\R^d} |\cF(f)(\bsomega)|^2(1+ \|\bsomega\|_2^2)^\sigma d\bsomega < \infty \right\},
\]
where $\cF(f)$ is the usual Fourier transform 
\[
  \cF(f)(\omega) = (2\pi)^{-d/2} \int_{\R^d} f(\bx) e^{-i \bx^T \bsomega} d\bx.
\]

Before introducing local Sobolev spaces on subdomains of the unit sphere,
let us recall a few key definitions on Sobolev spaces defined on a given bounded domain
$D$ in $\R^d$. For a given non-negative integer $m$, the Sobolev space
$H^m(D)$ consist of all $f$ with weak derivatives $D^\alpha f \in L_2(D)$,
$|\alpha| \le m$. The semi-norms and norms are defined by
\[
|f|_{H^m(D)} = \left( \sum_{|\alpha|=m} \| D^\alpha f\|^2_{L_2(D)}\right)^{1/2}
\quad \text{and} \quad
\|f\|_{H^m(D)} = \left( \sum_{|\alpha| \le m} \| D^\alpha f\|^2_{L_2(D)}\right)^{1/2}.
\]
For $m\in {\mathbb N}_0$, $ 0< s <1$, the fractional Sobolev spaces $H^{m+s}(D)$
is defined to be the set of all $f$ for which the following semi-norm and norm
\begin{align*}
|f|_{H^{m+s}(D)} &:= \left(
\sum_{|\alpha| = m} 
\int_D \int_D \frac{|D^\alpha f(\bx) - D^\alpha f(\by)|^2}{ \| \bx - \by\|^{d+2s}_2 }
\right)^{1/2} \\
\|f\|_{H^{m+s}(D)} &:= (\|f\|^2_{H^m(D)} + |f|_{H^{m+s}(D)})^{1/2}
\end{align*}
are finite. 

Let $\Omega \subset \mS^d$ be an open connected set with sufficiently smooth boundary. 
In order to define the spaces on $\Omega$, let
$
  D_k = \psi_k( \Omega \cap U_k)\text{ for } k=1,2.
$
The local Sobolev space $H^\sigma(\Omega)$ is defined to be the set
\[
 \left\{f \in L_2(\Omega) : \pi_k(\chi_k f)|_{D_k} \in H^\sigma(D_k)
 \text{ for } k=1,2, \; D_k \neq \emptyset\right\},
\]
which is equipped with the norm
\be\lb{def:local Sob}
\Vert f \Vert_{H^\sigma(\Omega)} =
\left( \sum_{k=1}^2 \|\pi_k(\chi_k f)|_{D_k}\|^2_{H^\sigma(D_k)}\right)^{1/2}
\ee
where, if $\Omega=\emptyset$, then we adopt the convention that
$\|\cdot\|_{H^\sigma(D_k)} = 0$.

We observe, following \cite{HubMor04}, that there exists a positive
constant $C_{\cA}$, depending on $\cA$ and the partition of unity
$\{\chi_1,\chi_2\}$, such that the geodesic distance of supp$\;\chi_k$
from the boundary of $U_k$ is strictly greater than $C_{\cA}$. A spherical
cap $G(\bz,\theta)$ with $\theta < C_{\cA}/3$ will have its closure being
a subset of at least one of the open subsets $U_1$ or $U_2$, defined by
\eqref{def_Ui}, and if the cap $G(\bz,\theta)$ is not a subset of one of
these subsets, say $U_2$, then its intersection with supp$\;\chi_2$ must
be empty.

Now we state an extension theorem for a spherical cap on
the sphere.
\begin{theorem}[Extension operator]\label{extension}
Let $\Omega=G(\bz,\theta)$ be a spherical cap for some $\bz \in \mS^d$ and
$\theta < C_{\cA}/3$. There is an extension operator
$E: H^\nu(\Omega) \rightarrow H^\nu(\mS^{d})$ for all $ \nu \ge 0$, 
with $E$ independent of $\nu$, such that
\begin{enumerate}
\item $Ef |_{\Omega} = f$ for all  $f\in  H^\nu(\Omega)$,
\item $\| Ef \|_{H^\nu(\mS^{d})}
 \le C_\nu  \|f\|_{H^\nu(\Omega)}$.
\end{enumerate}
\end{theorem}
\begin{proof}
The case of $\nu$ being an integer was proved in \cite[Theorem
4.3]{HubMor04}. The framework for the case of fractional order $\nu$ is
also available in \cite{HubMor04} even if the explicit statement is not
given there. For the sake of completeness, we give the proof here.

When $\nu$ is not an integer, let $k$ be the non-negative integer for which
$\nu = k+s$, with $s \in (0,1)$. By \cite[Theorem 4.3]{HubMor04},
there is an extension operator which maps $H^{k+i}(\Omega)$ to
$H^{k+i}(\mS^d)$ and there are constants $C_{k,i}$ for $i=0,1$ so that
\begin{align*}
  \|E f\|_{H^{k+i}(\mS^{d})} &\le C_{k,i}  \|f\|_{H^{k+i}(\Omega)}, \qquad i=0,1.
\end{align*}
Using the \emph{operator interpolation property} (see \cite{Tri78}) we conclude
that $E$ is a bounded linear map from $H^\nu(\Omega)$ to $H^\nu(\mS^d)$
and
\[
  \| Ef \|_{H^\nu(\mS^{d})}
 \le C_{k,0}^{1-s} C_{k,1}^{s}  \|f\|_{H^\nu(\Omega)}.
\]
Property 1) follows from the fact that $H^{\nu}(\mS^d) \subset H^k(\mS^d)$
and $H^{\nu}(\Omega) \subset H^k(\Omega)$.
\end{proof}
\subsection{Positive definite kernels on the unit sphere}
A continuous function $\Phi:\mS^d \times \mS^d \rightarrow \R$ we call a
\emph{positive semi-definite kernel}~\cite{Sch42,XuChe92} on $\mS^d$ if it
satisfies the following conditions:
\begin{itemize}
\item[(i)] $\Phi$ is continuous,
\item[(ii)] $\Phi(\bx,\by) = \Phi(\by,\bx)$ for all $\bx,\by \in \mS^d$,
\item[(iii)] For any set of distinct scattered points
    $X=\{\bx_1,\ldots,\bx_K\}\subset \mS^d$, the symmetric $K\times
    K$-matrix $[\Phi(\bx_p,\bx_q)]$ is positive semi-definite.
\end{itemize}
We call $\Phi$ positive definite if the matrix is positive definite.

We will work with a zonal kernel~$\Phi$ defined in terms of a univariate
function~$\phi:[-1,1] \rightarrow \R$ by
\begin{equation}\label{equ:Phi res}
\Phi(\bx,\by)=\phi(\bx\cdot \by)
\quad\text{for all $\bx$, $\by\in\mS^d$.}
\end{equation}
Following M\"uller~\cite{Mul66}, let $P_\ell(t)$ denote the Legendre
polynomial of degree~$\ell$ for~$\R^{d+1}$, and expand $\phi(t)$ in a
Fourier--Legendre series

\begin{equation}\label{equ:defphi}
\phi(t) = \frac{1}{\omega_d}
          \sum_{\ell=0}^\infty
           N(d,\ell)\, \wh{\phi}(\ell) P_\ell(t).
\end{equation}
Due to the addition formula for spherical harmonics \cite[page 10]{Mul66}
\begin{equation}\label{add_thm}
\sum_{k=1}^{N(n,\ell)}
    Y_{\ell, k}(\bx) Y_{\ell, k}(\by) =
         \frac{N(d,\ell)}{\omega_d} P_\ell(\bx \cdot \by),
\end{equation}
the kernel $\Phi$ can be represented as
\[
\Phi(\bx,\by) = \sum_{\ell=0}^\infty \sum_{k=0}^{N(d,\ell)}
\widehat{\phi}(\ell) Y_{\ell,k}(\bx) Y_{\ell,k}(\by),
\quad \bx,\by \in \mS^d.
\]
and since $P_\ell(1)=1$ we find that
\begin{equation}\label{eq: Phi H^tau}
\|\Phi(\bx,\cdot)\|_{H^\sigma(\mS^d)}^2=\frac{1}{\omega_d}\sum_{\ell=0}^\infty
        (1+\ell)^{2\sigma} (\wh{\phi}(\ell))^2 N(d,\ell),
        \quad\text{for all $\bx\in\mS^d$.}
\end{equation}

Chen et al.~\cite{CheMenSun03} proved that the kernel $\Phi$ is
positive definite if and only if $\wh{\phi}(\ell)\ge 0$ for all $\ell\ge
0$ and $\wh{\phi}(\ell)> 0$ for infinitely many even values of~$\ell$ and
infinitely many odd values of~$\ell$; see also Schoenberg~\cite{Sch42} and
Xu and Cheney~\cite{XuChe92}. Here, we assume there is a $\sigma>d/2$ and
there are positive constants $c_1$~and $c_2$ such that
\begin{equation}\label{cond:what}
c^2_1(1+\ell)^{-2\sigma}\le \wh{\phi}(\ell)\le c_2^2(1+\ell)^{-2\sigma},
\quad \text{for all $\ell\ge 0$},
\end{equation}
hence, $\Phi$ is positive definite. Also, since
$N(d,\ell)=O(\ell^{d-1})$ as~$\ell\to\infty$, the
sum~\eqref{eq: Phi H^tau} is finite for each
fixed~$\bx\in\mS^d$. Thus the function~$\by\mapsto\Phi(\bx, \by)$ belongs
to~$H^\sigma(\mS^d)$. Moreover, this function is continuous
by the Sobolev imbedding theorem.

The reproducing kernel Hilbert space (RKHS)
(also called the \emph{native space}) induced by $\Phi$ is defined to be
\begin{equation}\label{def:native}
\cN_\Phi = \left\{ f \in L_2(\mS^2):
  \|f\|^2_{\Phi} := \sumlk \frac{|\widehat{f}_{\ell, k}|^2}{\widehat{\phi}(\ell)}
   <\infty \right\}.
\end{equation}
Alternatively, $\cN_\Phi$ is the completion of
span$\{\Phi(\cdot,\bx): \bx \in \mS^d\}$ with respect to the norm $\|\cdot\|_{\Phi}$.
The norm is associated with the following inner product
\[
  (f,g)_\Phi = \sumlk \frac{ \wh{f}_{\ell, k} \wh{g}_{\ell, k}}
                          {\whphi(\ell)}, \qquad  f,g \in \cN_\Phi.
\]
The kernel $\Phi$ has the reproducing property
with respect to this inner product, that is
\begin{equation}\label{rep}
 f(\bx) =  (f,\Phi(\cdot,\bx))_\Phi, \qquad \bx\in\mS^d, f\in \cN_\Phi.
\end{equation}
It follows from \eqref{cond:what} that the norms in $H^{\sigma}(\mS^d)$
and $\cN_\Phi$ are equivalent.

\subsection{Kernels defined from radial basis functions}
Let $\Pi:\R^{d+1}\rightarrow \R$ be a compactly supported radial basis function (RBF)
with associated RKHS $H^{\tau}(\R^{d+1})$ with
$\tau > (d+1)/2$. Examples of such RBFs are the Wendland functions (see \cite{Wen05}).

By restricting the function $\Pi$ to the unit sphere $\mS^d \subset
\R^{d+1}$, we have a positive definite, zonal kernel on the unit sphere
\[
 \Phi(\bx,\by) = \Pi(\bx-\by), \quad \bx,\by \in \mS^d.
\]

\begin{lemma}[Native spaces]\label{native}
Let $\Pi : \R^{d+1} \to \R$ be a positive definite function with native
space $\cN_\Pi(\R^{d+1}) = H^{\tau}(\R^{d+1})$ with $\tau> (d+1)/2$. Then
$\mathcal {N}_\Phi (\mS^{d})= H^{\sigma}(\mS^{d})$ with $\sigma=\tau-\frac{1}{2}$.
\end{lemma}
\begin{proof}
Using \cite[Proposition 4.2]{NarSunWar07}, we deduce that
\[
c (1+\ell)^{-2\sigma} \le \wh{\phi}(\ell) \le C (1+\ell)^{-2\sigma}.
\]
So the result follows from the definition of the Sobolev
spaces~\eqref{def:Sob S} and the native spaces \eqref{def:native} on
$\mS^d$.
\end{proof}

For a given $\delta>0$, we define the scaled version $\Phi_\delta$
of the kernel $\Phi$ by
\begin{equation}\label{scaledPhi}
\Phi_\delta(\bx,\by) = \delta^{-d} \Pi((\bx-\by)/\delta).
\end{equation}
We can expand $\Phi_\delta$ into a series of spherical harmonics
\[
  \Phi_\delta(\bx,\by) =
  \sumlk \wh{\phi_\delta}(\ell) Y_{\ell, k}(\bx) Y_{\ell, k}(\by),
\]
in which the Fourier coefficients satisfy the following
condition (see \cite[Theorem 6.2]{LeGSloWen10})
\begin{equation}\label{cond:whatdel}
 c_1^2(1+\delta\ell)^{-2\sigma} \le \wh{\phi_\delta}(\ell)
 \le c_2^2 (1+\delta\ell)^{-2\sigma},
\end{equation}
with the coefficients $c_1$ and $c_2$ from \eqref{cond:what} possibly
relaxed so that \eqref{cond:whatdel} holds for all $0 < \delta \le 1$.

For a function $f \in H^\sigma(\mS^d)$, we define the
norm corresponding to the scaled kernel $\Phi_\delta$ by

\begin{equation}\label{def:scaledPhinorm}
\|f\|_{\Phi_\delta} =
\left( \sumlk \frac{ |\wh{f}_{\ell, k}|^2 }
              {\widehat{\phi_\delta}(\ell)} \right)^{1/2},
\end{equation}
and the corresponding inner product is
\begin{equation}\label{def:scaledPhi inner}
  (f,g)_{\Phi_\delta} =
   \sumlk \frac{ \wh{f}_{\ell, k} \wh{g}_{\ell, k}}
    {\wh{\phi_\delta}(\ell)}, \qquad f,g \in \cN_\Phi.
\end{equation}
Clearly the norms $\Vert\cdot\Vert_{\Phi_\delta}$ for different $\delta$
are all equivalent, as given in the following lemma.

\begin{lemma}[Norm-equivalence] \label{norm}
Let $\Pi:\R^{d+1}\to \R$ be a reproducing
kernel of $H^\tau(\R^{d+1})$ with $\tau>(d+1)/2$. Let
$\Phi_\delta(\bx,\by)=\delta^{-d}\Pi((\bx-\by)/\delta)$ with
$\bx,\by\in \mS^{d}$. Then with $\sigma=\tau-1/2$,
\[
c_1 \|u\|_{\Phi_\delta} \le \|u\|_{H^\sigma(\mS^{d})} \le c_2
\delta^{-\sigma}\|u\|_{\Phi_\delta}
\]
for all $u\in H^\sigma(\mS^{d})$.
\end{lemma}
\begin{proof}
See \cite[Lemma 3.1]{LeGSloWen10}.
\end{proof}

From \eqref{def:scaledPhi inner} it follows that the reproducing property
\eqref{rep} extends to general $\delta$, that is
\begin{equation}\label{repdelta}
f(\bx) = (f,\Phi_\delta(\bx,\cdot))_{\Phi_\delta}, \quad \bx \in \mS^d,
\;\; f \in H^{\sigma}(\mS^d).
\end{equation}
\section{From global to local multiscale RBF interpolation}
\label{sec:zoom}
In this section we first consider RBF interpolation with a single scale,
then turn to multiscale interpolation, both global and local.
\subsection{Interpolation using spherical RBFs}
Let $X=\{\bx_1,\ldots,\bx_N\} \subset \Omega \subseteq \mS^d$ be a finite
set of distinct points on $\Omega$. We define the mesh norm $h_{X,\Omega}$
and the separation radius $q_X$ of this point set by
\[
  h_{X,\Omega} = \sup_{\bx \in \Omega} \min_{\bx_j \in X} \theta(\bx,\bx_j),
  \quad
  q_X = \frac{1}{2} \min_{i\ne j} \theta(\bx_i,\bx_j),
\]
where $\theta(\bx,\by) = \cos^{-1}(\bx \cdot \by)$ is the geodesic
distance on $\mS^d$. If $\Omega$ is a proper subset of $\mS^d$ then we say
that $h_{X,\Omega}$ is a local mesh norm. If $\Omega=\mS^d$ then the mesh
norm is global, and we simply write $h_X$.

We define the interpolation operator $I_{X,\delta}$ associated with
the set $X$ and the kernel $\Phi_\delta$ by
\begin{equation}\label{def:Ialpha}
  I_{X,\delta} f(\bx) =
   \sum_{j=1}^N b_j \Phi_{\delta}(\bx,  \bx_j),
   \quad  I_{X,\delta} f(\bx_j) = f(\bx_j) \mbox{ for all } \bx_j \in X.
\end{equation}
If $\delta=1$ then we simply write $I_{X} f$ instead of $I_{X,1} f$.
From the interpolation condition and \eqref{repdelta} we deduce that
\[
   (f - I_{X,\delta} f,\Phi_\delta(\cdot,\bx_j))_{\Phi_\delta} =
   f(\bx_j) - I_{X,\delta} f(\bx_j) = 0,
   \text{ for all } \bx_j \in X.
\]
Hence $I_{X,\delta} f$ is the orthogonal projection of $f$ into
span$\{\Phi_\delta(\cdot,\bx_j):\bx_j \in X\}$, from which it
follows that
\begin{equation}\label{pythagore}
\| f - I_{X,\delta} f\|_{\Phi_\delta} \le \|f\|_{\Phi_\delta}.
\end{equation}
From Lemma~\ref{norm} we then have
\begin{equation}\label{orthoprop}
\|f - I_{X,\delta} f\|_{H^\sigma(\mS^d)} \le c_2 \delta^{-\sigma} \|f\|_{\Phi_\delta}.
\end{equation}
\begin{lemma}[Zeros Theorem]\label{zeros} Let $\Omega \subseteq \mS^d$
be either an open connected region with Lipschitz
boundary or $\Omega=\mS^{d}$. Assume that a finite set $X\subset\Omega$ has
a sufficiently small (local) mesh norm $h_{X,\Omega}$. Then,
for any function $u\in H^\sigma(\Omega)$, $\sigma>d/2$, with $u|_X =0$,
for all $0\le \nu \le \sigma$ we have
\[
\|u\|_{H^\nu(\Omega)} \le C h_{X,\Omega}^{\sigma-\nu}
\|u\|_{H^\sigma(\Omega)}.
\]
\end{lemma}
\begin{proof}
For $\Omega \subset \mS^d$ an open and connected set with Lipschitz
boundary, the proof follows from
the zeros lemma for Lipschitz domains
on a Riemannian manifold in \cite[Theorem A.11]{HanNarWar12}.
The case $\Omega = \mS^d$ was proved earlier in
\cite{LeGNarWarWen06}.
\end{proof}
\begin{theorem}
Let $\Omega \subseteq \mS^d$ be either a spherical cap that satisfies the conditions
of Theorem~\ref{extension} or $\Omega=\mS^d$. Assume that a finite set
$X\subset\Omega$ has a sufficiently small (local) mesh norm $h_{X,\Omega}$. Then,
\[
\|f - I_{X,\delta} f\|_{L_2(\Omega)} \le
C \delta^{-\sigma} h_{X,\Omega}^{\sigma}\|f\|_{H^\sigma(\Omega)}.
\]
In particular, when $\delta = 1$, we have
\[
\|f - I_{X} f\|_{L_2(\Omega)} \le
Ch^{\sigma}_{X,\Omega}\|f\|_{H^\sigma(\Omega)}.
\]
\end{theorem}
\begin{proof}
Let $u:=f-I_{X,\delta} f$, then $u|_X = 0$. Using Lemma~\ref{zeros}, we
have
\[
\|f - I_{X,\delta} f\|_{L_2(\Omega)} \le
Ch^{\sigma}_{X,\Omega}\|f - I_{X,\delta} f\|_{H^\sigma(\Omega)}.
\]
If $\Omega$ is a spherical cap our assumptions on $\Omega$ allow us to
extend the function $f\in H^\sigma(\Omega)$ to a function $Ef \in
H^\sigma(\mS^d)$. Moreover, since $X \subset \Omega$ and $Ef|_{\Omega} =
f|_{\Omega}$, the interpolant $I_{X,\delta} f$ coincides with the
interpolant $I_{X,\delta} (Ef)$ on $\Omega$. Therefore,
\begin{align*}
\|f-I_{X,\delta}f\|_{H^\sigma(\Omega)} &=
\|Ef-I_{X,\delta}(Ef)\|_{H^\sigma(\Omega)}
\le C\|Ef-I_{X,\delta}(Ef)\|_{H^\sigma(\mS^d)} \\
&\le  C \delta^{-\sigma}\|Ef\|_{\Phi_\delta}
\;\;\qquad\qquad\text{ by \eqref{orthoprop}}\\
&\le C\delta^{-\sigma}\|Ef\|_{H^\sigma(\mS^d)}
\quad\qquad\text{ by Lemma~\ref{norm}} \\
&\le C \delta^{-\sigma}\|f\|_{H^\sigma(\Omega)}
\qquad\qquad\text{ by Theorem~\ref{extension}}.
\end{align*}
The case $\Omega = \mS^d$ was proved in
\cite[Theorem 3.2]{LeGSloWen10}.
\end{proof}

\subsection{The global and local multiscale algorithm}
Suppose $X_1, X_2, \ldots, X_m \subset \mS^d$ is a sequence of finite
point sets with mesh norms $h_1,h_2,\ldots, h_m$ respectively. The mesh
norms are assumed to satisfy $h_{j+1} \approx \mu h_j$ for some fixed $\mu
\in (0,1)$. After that, suppose $X_{m+1},X_{m+2},\ldots,X_n \subset
\Omega$ is a sequence of point sets with local mesh norms
$h_{m+1,\Omega},\ldots, h_{n,\Omega}$, where $\Omega \subset \mS^d$ is
some open connected subset.  In future we will write 
$h_j$ for $h_{X_j,\Omega}$
for all $j=1,\ldots,n$.

Let $\delta_1,\delta_2,\ldots$ be a decreasing sequence of positive real
numbers defined by $\delta_j = \nu h_j$ for some $\nu > 0$. Taking the
scale proportional to the mesh norm in this way is desirable for both
numerical stability and efficiency, since the sparsity of the
interpolation matrix is maintained. For every $j=1,2,\ldots$ we define the
scaled SBF $\Phi_j := \Phi_{\delta_j}$, and also define the scaled
approximation space $W_j = \text{span} \{\Phi_j(\cdot,\bx): \bx \in
X_j\}$.

We start with a widely spread set of points on the global scale and use a
basis function with scale $\delta_1$ to recover the global behavior of the
function $f$ by computing $f_1=s_1: = I_{X_1,\delta_1} f$. The error, or
residual, at the first step is $e_1 = f-f_1$. To reduce the error, at the
next step we use a finer set of points $X_2$ and a finer scale $\delta_2$,
and compute a correction $s_2 = I_{X_2,\delta_2} e_1$ and a new
approximation $ f_2 = f_1 + s_2$, so that the new residual is $e_2 = f-f_2 = e_1
- I_{X_2,\delta_2} e_1$; and so on. After $m$ global steps we switch to
local refinement, i.e. from step $(m+1)$ onwards the set $X_{m+1}$ is localized
to a small region $\Omega$ on the sphere, and the new correction $s_{m+1}$
is constructed from the local space $W_{m+1}$ and the new approximation is
$f_{m+1} = f_m+ s_{m+1}$. The multiscale algorithm then is continued for
a further local $n-m$ steps.

\begin{algorithm}[!htbp]
\KwData{Right hand side $f$, number of levels $n$}
\Begin{
  Set $f_0 = 0$, $e_0 = f$. \\
  \For{$j=1,2,\ldots,n$}{
    Determine the (global or local) interpolant $s_j \in W_j$ to $e_{j-1}$\\
    Set $f_j = f_{j-1} + s_j$.\\
    Set $e_j = e_{j-1} - s_j$.
  }
}
\KwResult{Approximation solution $f_n \in W_1 + \cdots + W_n$}
\caption{Multiscale global/local algorithm}
\end{algorithm}

\noindent {\bf Remark} Clearly, we could continue the algorithm by choosing an even
smaller region $\Omega^\prime \subset \Omega$ and a sequence of point sets
$X_{n+1},X_{n+2},\ldots \subset \Omega^\prime$, and so on, until a desired
resolution is reached. For simplicity of presentation, we restrict
ourselves to 
the situation of one zooming-in region $\Omega$ 
in the subsequent error
analysis (though not in the numerical example). Extension of the
convergence theory to the general case is trivial.



We will show convergence for the scheme within a spherical cap
$\Omega$.

\begin{theorem}[Convergence for functions in $H^\sigma(\mS^d)$]
\label{thm:convergence1}
  Let $X_1,\ldots,X_m$ be a sequence of point sets in $\mS^d$
  and let $X_{m+1},\ldots,X_n$ be a sequence of point sets in
  $\Omega \subset \mS^d$, where $\Omega$ satisfies the requirements
  in Theorem~\ref{extension}. Assume that we are performing $m$ steps
  of the global multilevel algorithm on $\mS^{d}$ and then $n-m$
  steps of the local multilevel algorithm, localised to $\Omega$.
  Let $h_1,\ldots,h_m$ be the global mesh norms and
  $h_{m+1},\ldots,h_n$ be the local mesh norms of the sets
  $X_1,\ldots,X_m$ and $X_{m+1},\ldots,X_n$, respectively, and
  assume that, for some
  $\mu\in(0,1)$, $h_{j+1}=\mu h_j$, for each $j=1,\ldots,n-1$.

  Let $\Phi$ be a kernel generating $H^\sigma(\mS^d)$ and  let
  $\Phi_j := \Phi_{\delta_j}$ be defined by \eqref{scaledPhi}
  with scale factor $\delta_j = \nu h_j$ where
  $1/h_1 \ge \nu \ge \gamma/\mu \ge 1$ with a fixed $\gamma>0$.
Assume that the target function
  $f$ belongs to $H^\sigma(\mS^d)$.

 Then the algorithm converges in the $L_2(\Omega)$ sense linearly in the number
 of levels. To be more precise, there is a constant $C>0$ and a
 constant $\alpha>0$, where $\alpha<1$ for $\mu$ sufficiently small, such that
\[
\|f-f_n\|_{L_2(\Omega)} \le C \alpha^n \|f\|_{H^\sigma(\mS^d)}
\]
for all $f\in H^\sigma(\mS^{d})$.
\end{theorem}

The theorem is a generalisation of the main result in \cite{LeGSloWen10}.

In preparation for the proof of the theorem we first prove the following
technical lemma.

\begin{lemma}
\label{key native}
Let $e_j$ for $j=0,\ldots,n$ be as in Algorithm 1,
and let $E$ be the extension operator from $\Omega$ to
$\mS^d$ as defined in Theorem~\ref{extension}. Then
\begin{enumerate}
\item[(i)] $\|e_j\|_{H^\sigma(\mS^d)} \le C\delta^{-\sigma}_j \|e_{j-1}\|_{\Phi_j}$ for $j=1,\ldots,m,$
\item[(ii)] $\|e_{m+1}\|_{H^\sigma(\Omega)}
\le C\delta^{-\sigma}_{m+1} \|e_{m}\|_{\Phi_{m+1}}$ ,
\item[(iii)] $\|e_j\|_{H^\sigma(\Omega)} \le C\delta^{-\sigma}_j \|Ee_{j-1}\|_{\Phi_j}$ for $j=m+2,\ldots,n.$
\end{enumerate}
\end{lemma}
\begin{proof}
For $j=1,\ldots,m$ we have, using \eqref{orthoprop},
\[
\|e_j\|_{\HsigS} = \|e_{j-1}-I_{X_j,\delta_j} e_{j-1}\|_{\HsigS} \le C\delta_j^{-\sigma} \|e_{j-1}\|_{\Phi_j}.
\]
For $j=m+2,\ldots,n$ we have, by using the property of the extension operator
and \eqref{orthoprop}
\begin{align*}
\|e_j\|_{\HsigO} = \|e_{j-1}-I_{X_j,\delta_j} e_{j-1}\|_{\HsigO}
                           & = \|Ee_{j-1}-I_{X_j,\delta_j} Ee_{j-1}\|_{\HsigO} \\
                           &\le C\|Ee_{j-1}-I_{X_j,\delta_j} Ee_{j-1}\|_{\HsigS}\\
                           &\le C\delta_j^{-\sigma} \|E e_{j-1}\|_{\Phi_{j}}.
\end{align*}
For the intermediate case, when $j=m+1$, we avoid the extension operator by arguing as follows
\begin{align*}
\|e_{m+1}\|_{\HsigO} = \|e_m - I_{X_{m+1},\delta_{m+1}} e_m\|_{\HsigO}
 & \le C \|e_m - I_{X_{m+1},\delta_{m+1}} e_m\|_{\HsigS}\\
 & \le C\delta_{m+1}^{-\sigma} \|e_{m}\|_{\Phi_{m+1}}.
\end{align*}
\end{proof}

\begin{proofof}{Theorem~\ref{thm:convergence1}}
In the proof we use repeatedly the fact that $e_j|X_j = 0$, allowing us to
use the zeros theorem (Lemma~\ref{zeros}), and we also make essential use
of the extension theorem (Theorem~\ref{extension}) and
Lemma~\ref{key native}.

We start by noting that
\begin{eqnarray}
\|f-f_n\|_{L_2(\Omega)}=
\|e_n\|_{L_2(\Omega)}  &\le &  C h_n^{\sigma} \| e_n\|_{H^\sigma(\Omega)}
=C h^\sigma_n \|E e_n\|_{\HsigO} \nonumber \\
&\le& Ch^\sigma_n \|Ee_n\|_{H^\sigma(\mS^{d})} \nonumber \\
&\le& Ch^\sigma_n \delta^{-\sigma}_{n+1} \|Ee_n\|_{\Phi_{n+1}} \nonumber\\
&=& C \|Ee_n\|_{\Phi_{n+1}}, \label{ferror}
\end{eqnarray}
where in the second-last step we use Lemma~\ref{norm} with $\delta =
\delta_{n+1}$, and in the last step $h_n/\delta_{n+1}=1/(\mu\nu)\le
1/\gamma$.

The result will then follow by establishing the recursions
\begin{eqnarray}
\|e_j\|_{\Phi_{j+1}} &\le& \alpha \|e_{j-1}\|_{\Phi_j}, \quad j=1,\ldots,m \label{rec1}\\
\|Ee_{m+1}\|_{\Phi_{m+2}} &\le& \alpha \|e_{m}\|_{\Phi_{m+1}}, \label{rec}\\
\|E e_j\|_{\Phi_{j+1}} &\le& \alpha \|Ee_{j-1}\|_{\Phi_j}, \quad j=m+2,\ldots,n \label{rec2},
\end{eqnarray}
where $\alpha$ is some real number satisfying $0 < \alpha <1$.

The first of these is exactly as in \cite{LeGSloWen10}. We shall prove the
third recursion \eqref{rec2}, noting that \eqref{rec1} can then be
recovered by replacing $\Omega$ by $\mS^d$ and omitting the extension
operator $E$.

For $j=m+2,\ldots,n$ we have by definition of the norm
\eqref{def:scaledPhinorm} together with \eqref{cond:whatdel}
\begin{eqnarray*}
\|Ee_j\|^2_{\Phi_{j+1}}&=&
\sum_{\ell=0}^{ \infty} ~\sum_{k=1}^{N(d,\ell)}
\frac{|(\widehat{Ee_j})_{\ell,k}|^2}{\widehat{\phi_{\delta_{j+1}}}(\ell)} \\
&\le& C \sum_{\ell=0}^\infty
\sum_{k=1}^{N(d,\ell)}
~|(\widehat{Ee_j})_{\ell, k} |^2~(1+\delta_{j+1}\ell)^{2\sigma}\\
&=:&I_1 + I_2,
\end{eqnarray*}
where
\begin{align*}
I_1 & = C \sum_{\ell \le 1/\delta_{j+1}}
\sum_{k=1}^{N(d,\ell)}~|(\widehat{Ee_j})_{\ell, k} |^2~(1+\delta_{j+1}\ell)^{2\sigma}\\
& \le C 2^{2\sigma} \sum_{\ell=0}^\infty
\sum_{k=1}^{N(d,\ell)}~|(\widehat{Ee_j})_{\ell, k} |^2 = C \|E e_j\|^2_{L_2(\mS^d)}\\
&\le C \|e_j\|^2_{L_2(\Omega)}
\hspace{4.1cm} \text{ by Theorem~\ref{extension}}\\
&\le C h_j^{2\sigma} \|e_j\|^2_{\HsigO}
\hspace{3.45cm}\text{ by Lemma~\ref{zeros}}\\
&\le C \left(\frac{h_j}{\delta_j}\right)^{2\sigma} \|E e_{j-1}\|^2_{\Phi_j}
\hspace{2.4cm}\text{ by Lemma~\ref{key native} iii)} \\
& = C \nu^{-2\sigma} \|E e_{j-1} \|^2_{\Phi_j}
 \le C \mu^{2\sigma}\|E e_{j-1} \|^2_{\Phi_j},
\end{align*}
and
\begin{align*}
I_2 & = C \sum_{\ell  > 1/\delta_{j+1}}~\sum_{k=1}^{N(d,\ell)}
~|(\widehat{Ee_j})_{\ell, k} |^2~(1+\delta_{j+1}\ell)^{2\sigma}\\
& \le C (2\delta_{j+1})^{2\sigma} \sum_{\ell=0}^\infty ~\sum_{k=1}^{N(d,\ell)}
~|(\widehat{Ee_j})_{\ell, k} |^2~ \ell^{2\sigma}\\
&\le C \delta^{2\sigma}_{j+1} \|Ee_j\|^2_{\HsigS}
\le C\delta^{2\sigma}_{j+1} \|e_j\|^2_{\HsigO} \text{ by Theorem~\ref{extension}}\\
&\le C\left(\frac{\delta_{j+1}}{\delta_j}\right)^{2\sigma} \|E e_{j-1} \|^2_{\Phi_j}
\hspace{2.3cm}\text{ by Lemma~\ref{key native} iii)} \\
&= C \mu^{2\sigma} \|E e_{j-1} \|^2_{\Phi_j}.
\end{align*}
Thus we have proved $\|Ee_j\|_{\Phi_{j+1}}\le
C\mu^{-\sigma}\|Ee_{j-1}\|_{\Phi_j}$. With $\mu$ small enough, we can
choose $\alpha=C \mu^\sigma < 1$, so proving the recursion \eqref{rec2}.

Next, we discuss the switch over from global to local. We have
\[
e_{m+1}=e_{m} - I_{X_{m+1}, \delta_{m+1}} e_{m}, \quad \mbox{with} \; X_m \subset \mS^{d},
X_{m+1} \subset \Omega.
\]
As before, we decompose
\[
\| Ee_{m+1} \|^2_{\Phi_{m+2}} = I_1 + I_2.
\]
Now we have
\begin{eqnarray*}
I_1 &\le& C \|Ee_{m+1} \|^2_{L_{2}(\mS^d)} \\
&\le& C \| e_{m+1} \|^2_{L_{2}(\Omega)} \\
&\le& Ch_{m+1}^{2\sigma} \| e_{m+1} \|_{H^\sigma (\Omega)}^2
\hspace{1.2cm}\text{ by Lemma~\ref{zeros} }\\
&\le& C \left(\frac{h_{m+1}}{\delta_{m+1}}\right)^{2\sigma} \|e_{m} \|_{\Phi_{m+1}}^2
 \hspace{0.56cm}  \text{ by Lemma~\ref{key native} ii)}.
\end{eqnarray*}
The second term can be bounded by
\begin{eqnarray*}
I_2 &\le& C \delta_{m+2}^{2\sigma} \| Ee_{m+1} \|_{H^{\sigma} (\mS^{d})}^2
\le C \delta_{m+2}^{2 \sigma} \| e_{m+1} \|_{H^{\sigma} (\Omega)}^2\\
&\le&C \left(\frac {\delta_{m+2}}{\delta_{m+1}}\right)^{2\sigma}
\|e_{m}\|_{\Phi_{m+1}}^2 \quad\text{ by Lemma~\ref{key native} ii)}.
\end{eqnarray*}
Hence we find
\[
\|E e_{m+1} \|_{\Phi_{m+2}}^2 \le C \mu^{2\sigma} \| e_{m} \|_{\Phi_{m+1}}^2,
\]
and this can be no larger than $\alpha^2\|e_{m}\|_{\Phi_{m+1}}^2$ for
$\mu$ sufficiently small.

The first recursion \eqref{rec1} follows by the same proof if $E$ is
omitted and $\Omega$ is replaced by $\mS^d$.

Taken in the reverse order, the recursive steps \eqref{rec2}, \eqref{rec}
and \eqref{rec1} give
\begin{align*}
\|E e_n\|_{\Phi_{n+1}} &\le \alpha^{n-m-1} \|E e_{m+1}\|_{\Phi_{m+2} } \\
    & \le \alpha^{n-m} \|e_m\|_{\Phi_{m+1}} \\
    & \le \alpha^n \|e_0\|_{\Phi_1} = \alpha^n \|f\|_{\HsigS},
\end{align*}
which together with \eqref{ferror} proves the desired result.
\end{proofof}

The following result on the condition numbers of the matrices is
adapted from \cite[Theorem 7.3]{LeGSloWen10}.
\begin{theorem}
Assume that the conditions in Theorem \ref{thm:convergence1} hold, together with
\[
q_j \le h_j \le c_q q_j \mbox{ for }j=1,2,\ldots,n \mbox{ with }c_q > 1.
\]
There exists $C>0$ such that the condition number of the interpolation
matrices at each level of the multiscale approximation in Algorithm 1 are
bounded by
\[
\kappa \le  C, \quad j=1,\ldots,n.
\]
\end{theorem}

\section{Escaping the native space}\label{sec:escape}

In this section, our target function $f$ will be assumed to be in
$H^\beta(\mS^d)$ for some $\beta \in (d/2,\sigma)$. The extension of an
approximation result to spaces rougher than the native space is often
referred as ``escaping the native space''.

Let $K$ be the reproducing kernel of the Sobolev space
$H^{\beta+1/2}(\R^{d+1})$. We define the kernel $\Psi$ by restricting $K$
to the sphere,
\[
  \Psi(\bx,\by) = K(\bx-\by),\quad \bx,\by \in \mS^d.
\]
For $0<\delta \le 1$, the scaled version of $\Psi$ is defined by
\[
  \Psi_\delta(\bx,\by) =
  \delta^{-d} K((\bx-\by)/\delta),\quad \bx,\by \in \mS^d.
\]
It can be expanded into a series of spherical harmonics as
\begin{equation}\label{scaledPsi}
\Psi_\delta(\bx,\by) = \sum_{\ell=0}^\infty \sum_{k=0}^{N(d,\ell)}
 \widehat{\psi_\delta}(\ell) Y_{\ell, k}(\bx) Y_{\ell,k}(\by).
\end{equation}
It is known \cite[Lemma 2.1]{LeGSloWen12} that there
are positive constants $c_3,c_4$ independent of $\delta$ and $\ell$
so that
\begin{equation}
c^2_3(1+\delta \ell)^{-2\beta} \le \widehat{\psi_\delta}(\ell)
\le c^2_4(1+\delta \ell)^{-2\beta}, \quad \ell \ge 0.
\end{equation}

We can define the RKHS with the reproducing kernel
$\Psi_\delta$ and its norm $\|\cdot\|_{\Psi_\delta}$ as
in \eqref{def:native} and \eqref{def:scaledPhinorm}.
By Lemma~\ref{norm}, the norm $\| \cdot\|_{\Psi_\delta}$ defined on 
$\cN_{\Psi_\delta}$ is equivalent to $\|\cdot\|_{H^\beta(\mS^d)}$.

For the multiscale convergence theory, the sole thing that prevents us
from using the proof of Theorem~\ref{thm:convergence1} with $\Phi_\delta$
replaced by $\Psi_\delta$ is that a key stability property is missing: the
orthogonal projection property \eqref{orthoprop} no longer holds. We
therefore approximate a function in $\Psi_\delta$  by a polynomial (which
of course lies in all Sobolev spaces), and apply the orthogonal projection
property to that polynomial.

For a given smooth function $f$, the following lemma \cite[Lemma
4.3]{LeGSloWen12} asserts the existence of a spherical polynomial that
interpolates $f$ on a set of scattered points $X$ and, simultaneously, has
a $\Psi_ \delta$ norm comparable to that of $f$.
\begin{lemma}\label{lem:near best}
Let $f \in H^\beta(\mS^d)$, $\beta>d/2$ and let $X$ be a finite subset of $\mS^d$ with separation radius
$q_X$. Let $\delta \in (0,1]$ be given. There exists a constant $\kappa$, which depends only on $d$ and
$\beta$, such that if $L \ge \kappa \max\{\delta/q_X,1/\delta\}$, then there is a spherical polynomial
$p\in \cP_L$ such that $p_X = f|_X$ and
\[
  \|f-p\|_{\Psi_\delta} \le 5 \|f\|_{\Psi_\delta}.
\]
\end{lemma}

\noindent {\bf Remark} The dependence of the lower bound for $L$ on the
mesh radius $q_X$ in the last lemma makes it necessary to impose a weak
condition on $q_X$ in the following theorem.

\begin{theorem}[Convergence outside the native space]\label{thm:convergence2}
  Let $X_1,\ldots,X_m$ be a sequence of point sets in $\mS^d$
  and let $X_{m+1},\ldots,X_n$ be a sequence of point sets in
  $\Omega \subset \mS^d$ where $\Omega$ satisfies the requirements
  in Theorem~\ref{extension}. Assume that we are performing $m$ steps
  of the global multilevel algorithm on $\mS^{d}$ and then $n-m$
  steps of the local multilevel algorithm, localised to $\Omega$.

  Let the (global or local) mesh norms $h_1,\ldots,h_n$ and
  the separation radii $q_1,\ldots,q_n$ satisfy

  \begin{itemize}
  \item[(i)] $h_{j+1} = \mu h_j$ for $j=1,\ldots,n$ with $\mu \in (0,1)$,
  \item[(ii)] $q_j \le h_j \le c_q \sqrt{q_j}$ for $j=1,2,\ldots,n$.
  \end{itemize}

  Let $\Phi$ be a kernel generating $H^\sigma(\mS^d)$ and  let
  $\Phi_j := \Phi_{\delta_j}$ be defined by \eqref{scaledPhi}
  with scale factor $\delta_j = \nu h_j$ where
  $1/h_1 \ge \nu \ge \gamma/\mu \ge 1$ with a fixed $\gamma>0$.
  Let $\Psi$ be a kernel generating $H^\beta(\mS^d)$ with
  $\sigma>\beta>d/2$ and let $\Psi_j:=\Psi_{\delta_j}$ be the scaled version \eqref{scaledPsi}using
  the scale factor $\delta_j$. Assume that the target function
  $f$ belongs to $H^\beta(\mS^d)$.

  Then, Algorithm 1 converges in the $L_2(\Omega)$ sense linearly
  in the number of levels. To be more
  precise, there is a constant $C>0$ and a constant $\alpha>0$, which
  for $\mu$ sufficiently small is $<1$, such that
\[
\|f-f_n\|_{L_2(\Omega)} \le C \alpha^n \|f\|_{H^\beta(\mS^d)}
\]
for all $f\in H^\beta(\mS^{d})$.
\end{theorem}

Similarly to the case of Theorem~\ref{thm:convergence1}, the proof of the
theorem rests upon the following technical lemma.  But in this case the
proof is necessarily different, because the orthogonal projection property
\eqref{orthoprop} is not available.

\begin{lemma}
\label{key escape}
Let $e_j$ for $j=0,\ldots,n$ be as in Algorithm 1,
and let $E$ be the extension operator from $\Omega$ to
$\mS^d$ as defined in Theorem~\ref{extension}. 
Let the assumptions on $h_j$ and $q_j$ be satisfied
as in Theorem~\ref{thm:convergence2}.
Then
\begin{enumerate}
\item[(i)] $\|e_j\|_{H^\beta(\mS^d)} \le C\delta^{-\beta}_j \|e_{j-1}\|_{\Psi_j}$ for $j=1,\ldots,m,$
\item[(ii)] $\|e_{m+1}\|_{H^\beta(\Omega)}
\le C\delta^{-\beta}_{m+1} \|e_{m}\|_{\Psi_{m+1}}$,
\item[(iii)] $\|e_j\|_{H^\beta(\Omega)} \le C\delta^{-\beta}_j \|Ee_{j-1}\|_{\Psi_j}$ for $j=m+2,\ldots,n.$
\end{enumerate}
\end{lemma}

\begin{proof}

\noindent

We prove part iii) since part i) follows easily by replacing $\Omega$ by
$\mS^d$ and omitting the extension operator, and part (ii) is in an
obvious sense intermediate. See also the proof of \cite[Lemma
4.4]{LeGSloWen12}.

We use the extension operator to extend $e_{j-1}$ to $E e_{j-1}$ defined
on the whole sphere, for $j=m+2,\ldots,n$. Then, with $L_j := \lceil \kappa
\max\{ \delta_j/q_j,1/\delta_j\} \rceil$, by Lemma~\ref{lem:near best},
there is a polynomial $p\in \cP_{L_j}$ that interpolates and approximates $E
e_{j-1}$, in the sense that
\begin{equation}\label{p of f}
p|_{X_j}=Ee_{j-1}|_{X_j} \text{ and }
\|p-Ee_{j-1}\|_{\Psi_j} \le 5\|Ee_{j-1}\|_{\Psi_j}.
\end{equation}
We note that the RBF interpolant for $e_{j-1}$ coincides with the RBF
interpolant for $E e_{j-1}$ on $X_j$. Therefore,
\begin{eqnarray}
\|e_j\|_{H^\beta(\Omega)} &=&
\|e_{j-1} - I_{X_j,\delta_j} e_{j-1} \|_{H^\beta(\Omega)} \nonumber\\
&=& \|Ee_{j-1} - I_{X_j,\delta_j} Ee_{j-1} \|_{H^\beta(\Omega)} \nonumber\\
&\le & C\|Ee_{j-1} - I_{X_j,\delta_j} Ee_{j-1} \|_{H^\beta(\mS^d)} \nonumber\\
&\le& C\left(\|Ee_{j-1} - p\|_{H^\beta(\mS^d)} +
   \|p  - I_{X_j,\delta_j} Ee_{j-1} \|_{H^\beta(\mS^d)}\right). \label{ineq1}
\end{eqnarray}
The first term of \eqref{ineq1} can be bounded using Lemmas~\ref{norm}
and \ref{lem:near best},
\begin{equation}\label{1st est}
\| E e_{j-1} - p\|_{H^\beta(\mS^d)}
\le c_4 \delta_j^{-\beta} \|E e_{j-1} - p\|_{\Psi_j}
\le 5 c_4 \delta_j^{-\beta} \|E e_{j-1}\|_{\Psi_j}.
\end{equation}
For the second term, since $p|_{X_j} = Ee_{j-1}|_{X_j}$ the interpolant
$I_{X_j,\delta_j} E e_{j-1}$ is identical to $I_{X_j,\delta_j} p$, hence
by using Lemma~\ref{zeros} 
and \eqref{orthoprop}, we
have
\begin{align*}
\|p-I_{X_j,\delta_j} E e_{j-1}\|_{H^\beta(\mS^d)}
&=\|p-I_{X_j,\delta_j} p\|_{H^\beta(\mS^d)}
\le C h_j^{\sigma-\beta}\|p - I_{X_j,\delta} p\|_{H^\sigma(\mS^d)}\\
&\le C h_j^{\sigma-\beta}\delta_j^{-\sigma} \|p\|_{\Phi_j}
\le C \delta_j^{-\beta}\|p\|_{\Phi_j}.
\end{align*}
For the polynomial $p$ of degree $L_j$, using the definition \eqref{def:scaledPhinorm},
condition~\eqref{cond:whatdel} and the fact that $\beta < \sigma$, we have
\begin{align*}
\|p\|^2_{\Phi_j} &\le
C \sum_{\ell=0}^{L_j} \sum_{k=1}^{N(d,\ell)}
 (1+\delta_j \ell)^{2\sigma} |\widehat{p}_{\ell k}|^2 \\
 &\le C(1+\delta_j L_j)^{2(\sigma-\beta)}
 \sum_{\ell=0}^{L_j} \sum_{k=1}^{N(d,\ell)}
 (1+\delta_j \ell)^{2\beta} |\widehat{p}_{\ell k}|^2 \\
 &\le C \|p\|^2_{\Psi_j},
\end{align*}
where in the last step we used $L_j\le C/\delta_j$. 
(Since $h_j\le c_q \sqrt{q_j}$ and since $\delta_j=\nu h_j$, we see that $\delta_j/q_j \le c/\delta_j$ and hence 
$L_j\le C/\delta_j$).Thus, combining these
above estimates together with the fact that $\|p\|_{\Psi_j} \le 6\|E
e_{j-1}\|_{\Psi_j}$ we obtain
\begin{equation}\label{2nd est}
\|p-I_{X_j,\delta_j} E e_{j-1}\|_{H^\beta(\mS^d)}
\le C\delta_j^{-\beta} \|E e_{j-1}\|_{\Psi_j}.
\end{equation}
Combining \eqref{ineq1}, \eqref{1st est} and \eqref{2nd est}, we obtain
the desired result.
\end{proof}

\begin{proofof}{Theorem~\ref{thm:convergence2}}
The proof is identical to that for Theorem~\ref{thm:convergence1} once we
have established Lemma~\ref{key escape}: the only difference is that
$\sigma$ is replaced by $\beta$ and $\Phi_j$ by $\Psi_j$. We leave the
details to the reader.
\end{proofof}

\section{Numerical experiment}\label{sec:numerics}
In this section, we describe a numerical experiment that illustrates
the multiscale algorithm described in previous sections.

Let $\bp = (1/\sqrt{3},\; 1/\sqrt{3},\; 1/\sqrt{3})^T$ and $\bq =
(-0.7476,\;   0.5069, \;   0.4289)^T$ be two given points on $\mS^2$, and
let $\alpha :=\pi/12$ and $\rho := \pi/96$. Let $\Omega_1=G(\bq,\alpha)$
and $\Omega_2 = G(\bq,\rho)$ be concentric spherical caps centered at
$\bq$, with geodesic radii $\alpha$ and $\rho$ respectively. Note that the
successive areas of $\mS^2$, $\Omega_1$ and $\Omega_2$ are decreasing by a
factor of roughly $60$.

A point on $\mS^2$ is parametrized by polar coordinates $\theta, \phi$,
with
\[
\bx =(\sin\theta \cos \phi, \sin\theta \sin\phi, \cos\theta) \text{ for }
\theta \in [0,\pi] \text{ and } \phi \in [0,2\pi).
\]
Let $t = \cos^{-1}(\bp \cdot \bx)$ and let $s=\cos^{-1}(\bq \cdot \bx)$.

\begin{figure}[h]
\begin{center}
\includegraphics[scale=0.4]{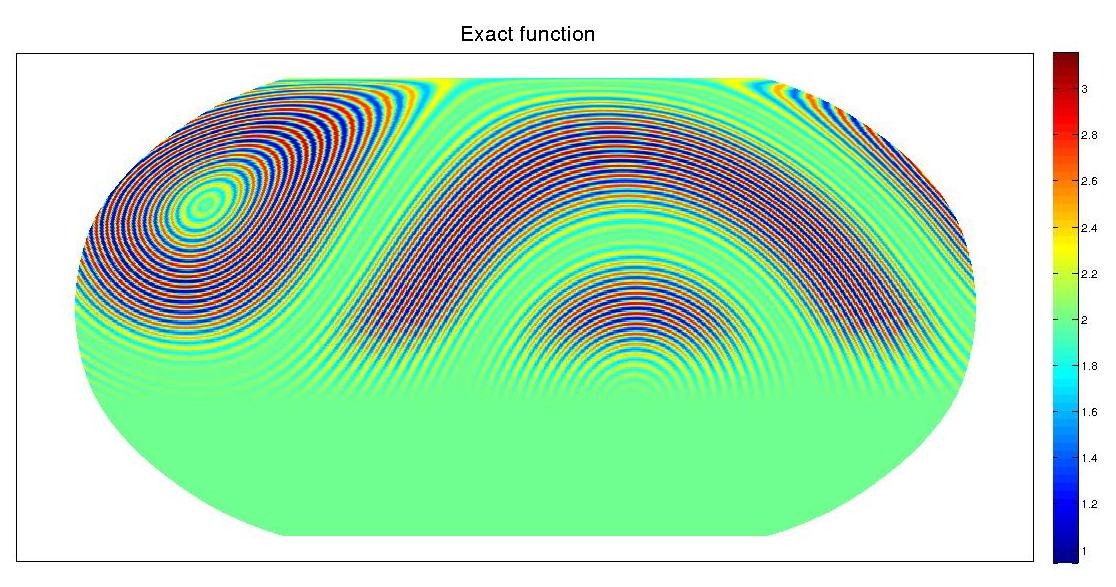}
\end{center}
\caption{Exact function from a global view}\label{fig:exactf}
\end{figure}

The target function $f$ is given by
\[
f(\bx) = 2 + \left[\sin t \cos (100t) + (1-3s/2\rho)^2_{+} \cos(2000\theta)\right]S(\theta),
\]
where $S(\theta)$ is a cubic spline which takes the values of $1$ for
$\theta \in [0,\pi/2]$ and $0$ for $\theta \in [2\pi/3,\pi]$. The function
$f$, shown in Figure~\ref{fig:exactf},  is designed to have both global
features and finer features.  On the global scale, the effect of the
spline multiplying the second term is that $f$ has the constant value $2$
below a latitude of $30^{o}$ south.  This feature was chosen because we
want to be sure that the approximation scheme approximates a constant
satisfactorily. (We remind the reader that approximating a constant with
compactly supported radial basis functions is non-trivial, specially if
the scale is comparable to the mesh norm.) The function $f$ also contains
a slow oscillation (seen in Figure~\ref{fig:exactf}) and a localized fast
oscillation inside the spherical cap $\Omega_2$, as shown in left panel of
Figure~\ref{fig:apprx9levelsVsExact}.  Note that the period of the
oscillation, given by the last term in the expression for $f$, corresponds
to approximately $20$ km if mapped to Earth's surface.  This finer
oscillation is too localized to be seen in Figure~\ref{fig:exactf}.

In the experiment we use $9$ multiscale levels, zooming in to the cap
$\Omega_1$ after three global levels, and zooming in again to the smaller
cap $\Omega_2$ after a further three levels. In the first three (global)
levels, the sets of points $X_1$, $X_2$, and $X_3$ are each centers of
equal area regions generated by a partitioning algorithm
\cite{RakSafZho94}. The number of points in each set $X_1, X_2, X_3$ is
increasing by a factor of $4$ (see Table~\ref{tab:err} below); the
sets are not nested. The sets $X_4,X_5$ and $X_6$ are also centers of
equal area regions, but the regions are partitioned from $\Omega_1$ rather
than the whole sphere. For simplicity of language we call levels $4$ to
$6$ the ``local'' levels. Similarly, $X_7, X_8$ and $X_9$ are the results
of partitioning $\Omega_2$ into equal area regions. We call levels $7$ to
$9$ the ``superlocal'' levels.  At every stage the scale is halved exactly
and the mesh norm halved approximately.  The parameter details for the
successive levels are given in Table \ref{tab:err}.

\begin{table}[h]
\begin{center}
\begin{tabular}{|c|c|c|c|c|c|}
\hline
Level &  $N$  & $\delta_j$  & $h_j$ & $\|e_j\|_{L_2(\Omega_2)}$ & $\kappa_j$ \\
\hline
$1$ &  $500$  & $1/4$       & 0.1129 & 4.24e-02 & 1.68\\
$2$ & $2000$  & $1/8$       & 0.0569 & 4.07e-02 & 1.68 \\
$3$ & $8000$  & $1/16$      & 0.0281 & 3.45e-02 & 1.69 \\
$4$ &  $500$  & $1/32$      & 0.0186 & 1.56e-02 & 3.25 \\
$5$ & $2000$  & $1/64$      & 0.0089 & 9.83e-03 & 3.39 \\
$6$ & $8000$  & $1/128$     & 0.0041 & 8.94e-03 & 3.30 \\
$7$ &  $500$  & $1/256$     & 0.0018 & 7.87e-03 & 3.24 \\
$8$ & $2000$  & $1/512$     & 0.0009 & 2.87e-03 & 3.37 \\
$9$ & $8000$  & $1/1024$    & 0.0005 & 7.97e-04 & 3.28 \\
\hline
\end{tabular}
\caption{Parameters and local errors using 9 levels of global and local
interpolation}\label{tab:err}
\end{center}
\end{table}

The RBF used in the experiment is the Wendland function
\[
\Pi(\|\bx\|) = (1-\|\bx\|)^4_{+} (4\|\bx\|+1)
\]
and its scaled version is
\[
  \Pi_\delta(\|\bx\|) = \delta^{-2}(1-\|\bx\|/\delta)^4_{+}
  (4\|\bx\|/\delta+1),
\]
where at level $j$, we set $\delta = \delta_j$. It is known that $\Pi$
generates $H^{3}(\R^3)$ (see \cite{Wen05}) and hence the kernel
$\Phi(\bx,\by) = \Pi(\bx-\by)$ for $\bx,\by \in \mS^2$ generates
$H^{5/2}(\mS^2)$ (see \cite{NarWar02}).

In Figure~\ref{fig:after3globalwithOmega_1}, we show the approximation
after the three global levels, using the point sets $X_1, X_2$ and $X_3$.
We also show on this figure the spherical cap $\Omega_1$, to show the
first region where we intend to zoom in.  At this stage it is clear
visually that the approximation scheme not yet resolved the slower
oscillations, but the broad features, including the constant value in
southern latitudes, are already apparent.

\begin{figure}[h]
\begin{center}
\includegraphics[scale=0.4]{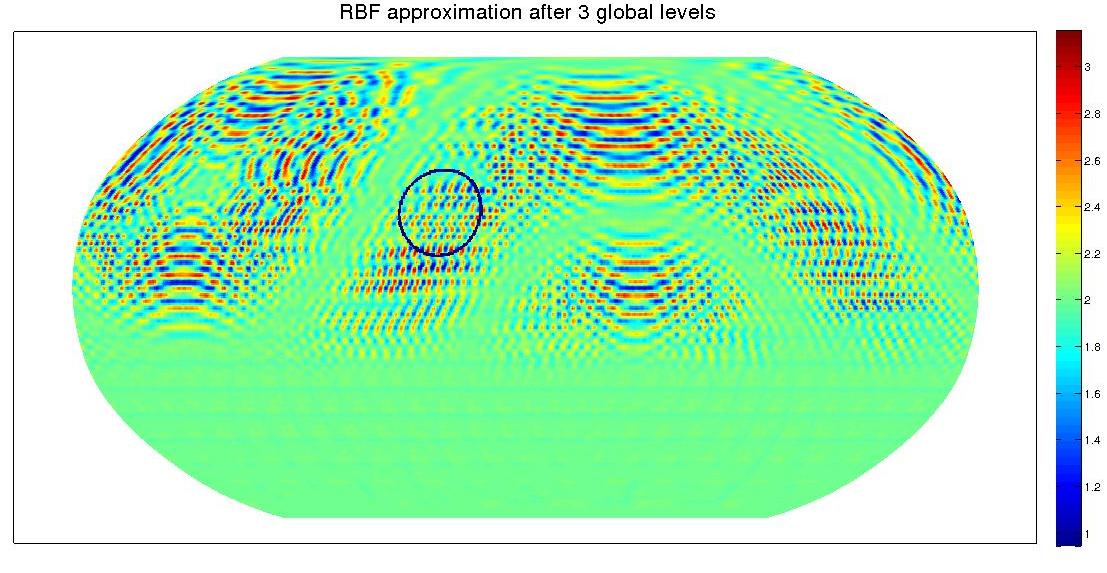}
\end{center}
\caption{The global view after three global levels, with the cap $\Omega_1$ shown}
\label{fig:after3globalwithOmega_1}
\end{figure}

In Figure~\ref{fig:after3glogal3localwithOmega_2}, we show the
approximation on the spherical cap $\Omega_1$ after $6$ levels ($3$ global
and $3$ local).  We also show the smaller spherical cap $\Omega_2$, inside
which it is clear that after 6 multiscale levels the slow oscillations
have largely been resolved but fine scale features have not.

\begin{figure}[h]
\begin{center}
\includegraphics[scale=0.4]{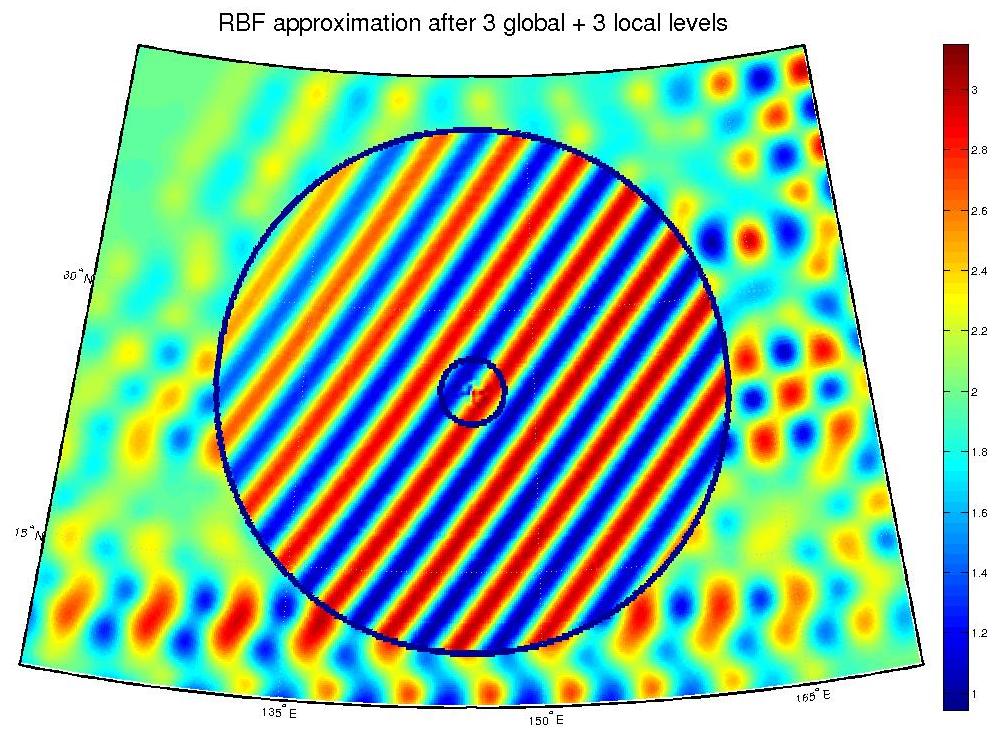}
\end{center}
\caption{The local view after $3$ global and $3$ local levels, showing
both the large cap $\Omega_1$ and the small (``superlocal'') cap $\Omega_2$}
\label{fig:after3glogal3localwithOmega_2}
\end{figure}

Finally, in Figure~\ref{fig:after3global3local3superlocal}, we show the
approximation on the small spherical cap $\Omega_2$ after $9$ levels ($3$
global, $3$ local and $3$ superlocal).  By this stage even the fine scale
features are well resolved.

\begin{figure}[h]
\begin{center}
\includegraphics[scale=0.4]{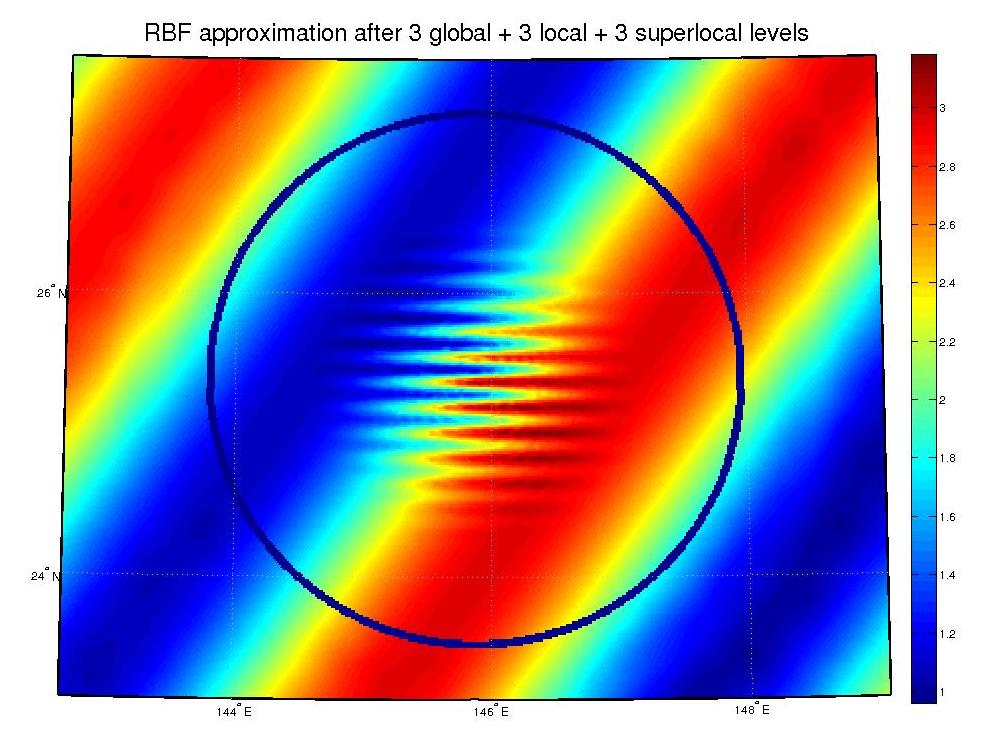}
\end{center}
\caption{Superlocal view of multiscale RBF approximation after $9$ levels, showing
small cap $\Omega_2$}
\label{fig:after3global3local3superlocal}
\end{figure}

For comparison, we carry out a more modest multiscale approximation in
which we use just the last three (superlocal) levels, and separately also
a single scale (`one-shot') approximation, in the second case using the
final scale $\delta=2^{-10}$ and the $8000$ sampling points inside the cap
$\Omega_2$. Poorer approximation quality of the one-shot interpolation can
be seen by eye in the right panel of
Figure~\ref{fig:apprx3levelsVsOneShot}. For the multiscale result in the
left panel of Figure~\ref{fig:apprx3levelsVsOneShot} that uses just the
last three levels the visual result is of intermediate quality: not as
good as the full multiscale result, but certainly better than the one-shot
result.

\begin{table}[h]
\begin{center}
\begin{tabular}{|c|c|c|c|c|c|}
\hline
Level &  $N$  & $\delta_j$  & $h_j$ & $\|e_j\|_{L_2(\Omega_2)}$ &  $\kappa_j$ \\
\hline
$1$ &  $500$  & $1/256$       & 0.0018 & 2.75e-02 &  3.24\\
$2$ & $2000$  & $1/512$       & 0.0009 & 1.49e-02 &  3.37\\
$3$ & $8000$  & $1/1024$      & 0.0005 & 9.18e-03 &  3.28\\
\hline
\end{tabular}
\caption{Local errors table when using multiscale approximation only at
the last 3 superlocal levels}\label{tab:err2}
\end{center}
\end{table}

\begin{figure}[h]
\begin{center}
\includegraphics[scale=0.55]{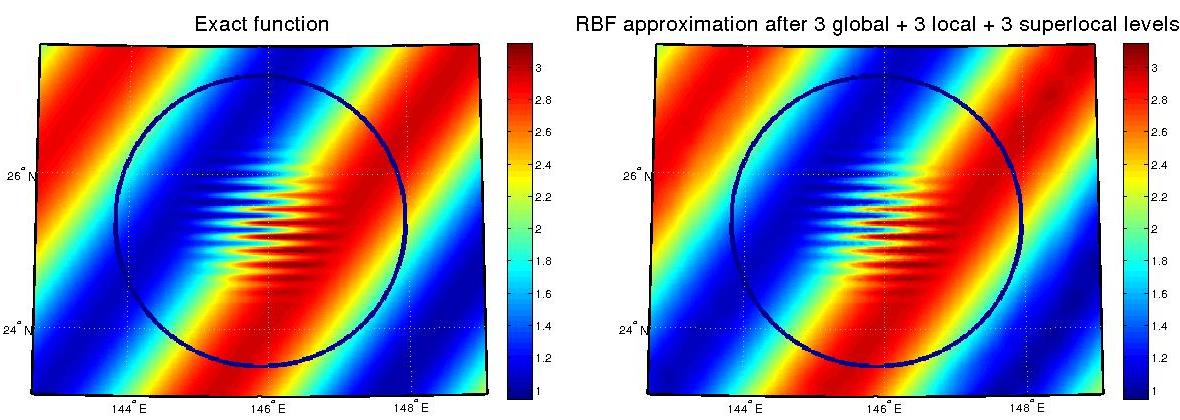}
\end{center}
\caption{Exact function (left) vs. multiscale RBF approximation
after $9$ levels (right)}
\label{fig:apprx9levelsVsExact}
\end{figure}

\begin{figure}[h]
\begin{center}
\includegraphics[scale=0.55]{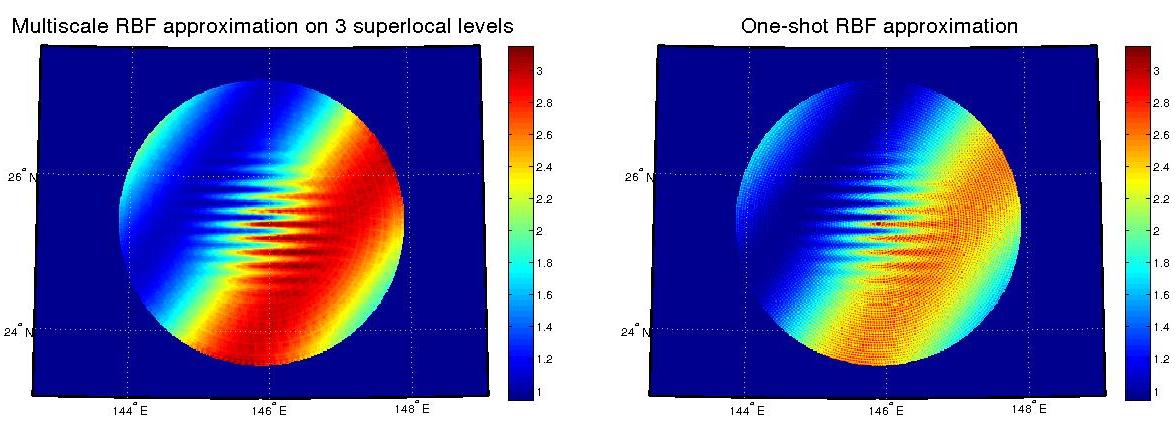}
\end{center}
\caption{Multiscale RBF approximation with $3$ superlocal levels (left) and
one-shot interpolation on level 9 (right)}
\label{fig:apprx3levelsVsOneShot}
\end{figure}

In Tables~\ref{tab:err} and \ref{tab:err2} approximate $L_2(\Omega_2)$
errors are given, in the first case for the full $9$-level multiscale
approximation, in the second case for the $3$-level superlocal version.
These were computed over a rectangular grid $\cG$ of size $1/64$ degree
times $1/64$ degree restricted to the spherical cap $\Omega_2$,
\[
   \|e_j\|_{L_2(\Omega_2)}:=
   \left(
   \frac{|\Omega_2|}{|\cG\cap \Omega_2|}
    \sum_{\bx(\theta,\phi) \in \cG \cap \Omega_2}
   |f(\theta,\phi) - f_j(\theta,\phi)|^2
   \right)^{1/2},
\]
where the area $|\Omega_2|$ of the cap $\Omega_2$ is included  so that the
computed quantity is an approximation to the $L_2(\mS^2)$ norm of the
error. With the grid $\cG$ as above the number of points in the cap
$\Omega_2$ is $|\cG\cap \Omega_2| = 50063$. The condition number of the
interpolation matrix at level $j$ is denoted by $\kappa_j$.

The $\|e\|_{L_2(\Omega_2)}$ error for the one-shot approximation is
$2.00e-02$, which is much larger than errors from the level $9$-level
multiscale approach, and also larger than the error from the $3$-level
multiscale approach in Table~\ref{tab:err2}.  Indeed, it is even an order
of magnitude larger than the approximate $L_2(\Omega_2)$ norm of the
function $f$ itself, which is $7.20e-03$.  The reason for this bad result
is that the one-shot approximation, with its relatively small scale
compared to the mesh norm, fails to resolve well even the slowly varying
background features -- witness the ``pepper and salt'' nature of the image
on the slowly varying part of the right-hand image in Figure~\ref{fig:apprx3levelsVsOneShot}.  
Even the $3$-level multiscale approximation is struggling to resolve the slowly varying background.

A final conclusion might be that the ``zooming in'' multiscale
approximation is successful at all levels.  It could be continued
indefinitely to smaller and smaller regions, giving a consistent
approximation scheme at all levels if the data is available.

\paragraph{Acknowledgement} The
authors gratefully acknowledge the support of the Australian Research
Council.

\end{document}